\documentclass{amsart}
\usepackage{graphicx}
\usepackage{graphicx}
\usepackage{color}
\usepackage{times}
\usepackage{cite}
\usepackage{enumerate,latexsym}
\usepackage{latexsym}
\usepackage{amsmath,amssymb}
\usepackage{graphicx}
\usepackage{amsthm}
\usepackage{verbatim}
\usepackage{hyperref}
\newtheorem{thm}{Theorem}[section]
\newtheorem{cor}[thm]{Corollary}

\newtheorem{lem}[thm]{Lemma}
\newtheorem{prop}[thm]{Proposition}
\theoremstyle{definition}

\theoremstyle{remark}

\numberwithin{equation}{section}

\newcommand{\set}[1]{\left\{#1\right\}}
\newcommand{\Real}{\mathbb R}
\newcommand{\eps}{\varepsilon}

\newcommand{\func}[1]{\ensuremath{\mathop{\mathrm{#1}}} }

\newcommand{\Div}[0]{\func{div}}

\newcommand{\xX}[0]{\mathbf{x}}

\newcommand{\hH}[0]{\mathbf{H}}

\newcommand{\nN}[0]{\mathbf{n}}

\title[A Uniqueness Property of high multiplicity tangent flows in 
dimension three]{A Remark on a uniqueness property of high multiplicity tangent flows in 
dimension three}
\author{Jacob Bernstein}
\address{Department of Mathematics, Johns Hopkins University, 3400 N. Charles Street, Baltimore, MD 21218}
\email{bernstein@math.jhu.edu}
\author{Lu Wang}
\address{Department of Mathematics, Johns Hopkins University, 3400 N. Charles Street, Baltimore, MD 21218}
\email{lwang@math.jhu.edu}
\thanks{The first author was partially supported by the EPSRC Programme Grant entitled ``Singularities of Geometric Partial 
Differential Equations'' grant number EP/K00865X/1 and by the NSF Grant DMS-1307953.  The second author was partially supported by the 
AMS-Simons Travel Grant}

\begin{document}
\begin{abstract}
 In this note, we combine the work of Ilmanen \cite{I} and of Colding-Ilmanen-Minicozzi \cite{CIM} to observe a 
uniqueness property for tangent flows at the first singular time of a smooth mean curvature flow of a closed surface in $\Real^3$. 
Specifically, if, at a fixed singular point, one tangent flow is a positive integer multiple of a shrinking $\Real^2, \mathbb{S}^1\times 
\Real$ or $\mathbb{S}^2$, then, modulo rotations, all tangent flows at the point are the same.  
\end{abstract}
\maketitle
\section{Introduction}
The \emph{mean curvature flow} is the negative gradient flow for area and is given by
\begin{equation}\label{eqn:MCF}
 \frac{d\xX}{dt} = \hH 
\end{equation}
where $\hH=-H\nN$ is the mean curvature vector of the map $\xX(t,\cdot):M \to M_t\subset\Real^{n+1}$, $\nN$ the unit normal of $M_t$ and 
$H=\Div \nN$ the mean curvature.  For compact hypersurfaces, this flow always develops singularities in finite time.  Using Huisken's 
monotonicity formula for the mean curvature flow (see \cite{HuiskenM1}, \cite{HuiskenM2}), one can show that at a singular point of the flow, a sequence of parabolic rescalings 
subconverge, in a measure theoretic sense, to a solution of the mean curvature flow that moves by homothety -- see \cite{ISing}, \cite{I} and 
\cite{WSing}.  Such flows are called \emph{tangent flows} and, in principle, depend on the choice of rescaling sequence.

In this short note, we combine the work of Ilmanen \cite{I} and Colding-Ilmanen-Minicozzi \cite{CIM} in order to observe a 
novel uniqueness property for tangent flows of high multiplicity at the first singular time of a smooth mean curvature flow of a closed 
surface in $\Real^3$. Specifically, if, at a point, one tangent flow is a positive integer multiple of a shrinking $\Real^2, 
\mathbb{S}^1\times \Real$ or $\mathbb{S}^2$, then all tangent flows at the point agree modulo rotations.  That is,
 \begin{thm} \label{thm:tangentflow}
 Let $M_t$, $0\leq t<T$, be a smooth embedded mean curvature flow of a closed surface in $\Real^3$.  Suppose that $(\xX,T)$ is a 
singular point of this flow and that $m[\Sigma_t]$ is a tangent flow at $(\xX, T)$.  If, after rotating $\Real^3$, 
$\Sigma_{-1}=\Real^k \times \mathbb{S}^{2-k}$ for $0\leq k \leq 2$, then all tangent flows at $(\xX, T)$ are, up to a rotation of 
$\Real^3$, equal to $m[\Sigma_t]$. 
\end{thm}
When the multiplicity is one, i.e., $m=1$,  Theorem \ref{thm:tangentflow} holds in all dimensions.  This follows  from the work of Brakke 
\cite{Brakke}, Huisken 
\cite{Huisken} and Colding-Ilmanen-Minicozzi \cite{CIM} -- for, respectively, $\Real^n$, $\mathbb{S}^n$ and $\mathbb{S}^k\times 
\Real^{n-k}$, $1\leq k \leq n-1$. 
In fact, in this case, more is now known -- i.e., that the tangent flows are unique. This is an immediate consequence of Brakke's 
regularity theorem \cite{Brakke} when the tangent flow is $\Real^n$ and is automatic for $\mathbb{S}^n$ by symmetry. It is a much harder 
problem when the tangent flow is cylindrical, and has been shown only very recently by Colding-Minicozzi 
\cite{CMUniq}. 

Theorem 
\ref{thm:tangentflow} is notable as little is known in general about tangent flows of higher multiplicity for the mean curvature flow (or 
even of tangent cones of higher multiplicity for minimal submanifolds). The well-known multiplicity 
one conjecture asserts that in 
the situation under consideration, all tangent flows are of multiplicity one -- see \cite[pg. 8]{I}.  In all dimensions, this is known to 
be true for mean convex mean curvature flows -- see White \cite{WhiteMC}, Sheng-Wang \cite{Sheng} 
and also the recent work of 
Andrews \cite{Andrews} and Haslhofer-Kleiner \cite{HK}.  However, even in dimension three, essentially nothing is known without the 
mean convexity assumption.

\section{Notation} 
On $\Real^{n+1}$, let us define the Gaussian weight to be
\[
 \Phi(\xX)=e^{-\frac{|\xX|^2}{4}}.
\]
For a hypersurface $\Sigma$ of $\Real^{n+1}$, 
 define $\mathcal{F}$, the $\Phi$ weighted area, to be
\begin{equation*}
 \mathcal{F}[\Sigma]=\int_{\Sigma} \Phi \; d \mathcal{H}^{n},
\end{equation*}
where $\mathcal{H}^n$ is the $n$-dimensional Hausdorff measure.
A hypersurface is said to be a \emph{self-shrinker} if it is a critical point for 
$\mathcal{F}$.
The Euler-Lagrange equation of $\mathcal{F}$ is
\begin{equation*}
 \mathbf{H}+\frac{\xX^\perp}{2}=0,
\end{equation*}
where $\xX^\perp$ is the normal component of the position vector $\xX$. 
A smooth solution, $\Sigma$, to this equation is called a \emph{self-shrinker} and has the property that, for $t<0$, the flow 
$\Sigma_t=\sqrt{-t} 
\Sigma$ is a solution to the mean curvature flow \eqref{eqn:MCF}.
For $0\leq k \leq n$, the hypersurfaces
\begin{equation*}
 \mathbb{S}^{n-k}\times\Real^{k}=\set{\sum_{i=k}^n x_{i+1}^2 = 2(n-k)}
\end{equation*}
are all self-shrinkers. 

Let
\begin{equation*}
 \mathcal{L}_\Sigma=\Delta_\Sigma -\frac{\xX}{2} \cdot \nabla_\Sigma 
\end{equation*}
be the natural drift Laplacian on hypersurfaces which arises from using the weight $\Phi$.
A straightforward computation gives that
\begin{equation*}
 L_\Sigma=\mathcal{L}_\Sigma+|A|^2 +\frac{1}{2}
\end{equation*}
is the linearization of the self-shrinker equation.

For a properly embedded hypersurface $\Sigma\subset \Real^{n+1}$, let $[\Sigma]=\mathcal{H}^n\llcorner \Sigma$ be the associated Radon 
measure.
Following \cite{CIM}, we introduce a Gaussian weighted distance on the space of Radon measures on $\Real^{n+1}$.  Namely, let 
$f_n$ be a countable dense subset of the unit ball in the space of continuous functions  
with compact support in $\Real^{n+1}$.  Given two Radon measures $\mu_1, \mu_2$ on $\Real^{n+1}$, set
\begin{equation}\label{eqn:RadonDist}
 d_V(\mu_1, \mu_2)=\sum_{k=1}^\infty 2^{-k}\left| \int f_k \Phi d\mu_1- \int f_k \Phi d\mu_2 \right|.
\end{equation}
It is straightforward to verify that $d_V$ is a metric on the space of Radon measures satisfying 
\begin{equation*}
 \mathcal{F}[\mu]:=\int \Phi d\mu<\infty
\end{equation*}
and that $\mu_i\to \mu$ in the weak* topology if and only if $d_V(\mu_i, \mu)\to 0$.

\section{A Rigidity Property of Smooth Shrinkers}
We first note a simple fact about the volume growth of self-shrinkers.
\begin{lem} \label{VolGrowLem}
  If $\Sigma^n\subset B_{2R} \subset
\Real^{n+1}$ is a smooth properly embedded self-shrinker and $2\sqrt{2n}\leq R_1<R_2\leq R$, then
\begin{equation*}
  \mathcal{H}^n(B_{R_2}\cap\Sigma ) \leq 2^{n+3} \frac{R_2^n}{R_1^n}  \mathcal{H}^n(B_{R_1}\cap\Sigma ).
 \end{equation*}
\end{lem}
\begin{proof}
 This follows immediately from standard volume estimates for mean curvature flow and the fact that $\Sigma$ moves by scaling -- see 
\cite[Proposition 4.9]{EckerBook}.
\end{proof}
\begin{cor} \label{BrakkeCor}
 There are constants $R=R(n)$ and $\delta=\delta(n)$ so that: If $\Sigma^n\subset\Real^{n+1}$ is a smooth properly embedded self-shrinker 
satisfying
\begin{equation*}
 \mathcal{H}^n(B_R\cap \Sigma)-\omega_n R^n \leq \delta
\end{equation*}
and
\begin{equation*}
  \int_{B_R\cap \Sigma} \Phi -\int_{B_R\cap \Real^n} \Phi \leq \delta,
 \end{equation*}
then $\Sigma$ is a hyperplane.
\end{cor}
\begin{proof}
 Let $\epsilon=\epsilon(n)>0$ be the constant in Brakke's theorem, that is, $\epsilon$ is chosen so
 \begin{equation*}
  \int_{\Sigma}\Phi < 1+\epsilon(n)
 \end{equation*}
implies that $\Sigma$ is a hyperplane -- see \cite{WhiteLocal}.  Pick $\delta=\epsilon/3$ and use Lemma \ref{VolGrowLem} to choose 
$R=R(n)$ so
\begin{equation*}
 \int_{\Sigma\backslash B_{R}} \Phi \leq \epsilon/3.
\end{equation*}
\end{proof}
We also note that on sufficently large scales shrinkers are unstable -- this is implicit in the proof of \cite[Theorem 0.5]{CM} -- we 
include a proof for completeness.
\begin{lem}\label{InstabLem}
 If $R\geq 2^{n+5}$ and $\Sigma^n\subset B_{R} \subset
\Real^{n+1}$ is a smooth properly embedded self-shrinker, then $\lambda(L_\Sigma)<0$. Here $\lambda(L_\Sigma)$ is the first 
Dirichlet eigenvalue of $L_\Sigma$.  
\end{lem}
\begin{proof}
Write $R_d=2^{n+4}$. Let 
\begin{equation*}
 \phi_{R_d}(\xX)=\left\{ \begin{array}{ll} 1 & \mbox{if $|\xX|<R_d/2$.}\\
                                  2-\frac{2}{R_d} |\xX| & \mbox{if $R_d/2\leq |\xX|<R_d$.}\\
                                   0 &  \mbox{if $|\xX|\geq R_d$.}
                \end{array} \right.
\end{equation*}
One has that,
\begin{align*}
 -\int_{\Sigma} \phi_{R_d} L_\Sigma \phi_{R_d} \Phi& \leq \int_{\Sigma} \left(|\nabla_\Sigma \phi_{R_d}|^2-\frac{1}{2} \phi_{R_d}^2\right) \Phi \\
 & \leq \frac{4}{R_d^2} e^{-\frac{R_d^2}{16}} \mathcal{H}^n(\Sigma\cap B_{R_d})-\frac{1}{2} e^{-\frac{R_d^2}{16}} \mathcal{H}^n(\Sigma \cap B_{R_d/2}) \\
 &\leq \left( \frac{2^{2n+5}}{{R_d}^2}-\frac{1}{2}\right) e^{-\frac{R_d^2}{16}} \mathcal{H}^n(\Sigma \cap B_{R_d/2})\\
 &< 0,
 \end{align*}
where the second to last inequality follows from Lemma \ref{VolGrowLem}.
This proves the claim by the variational characterization of the lowest eigenvalue of $L_\Sigma$.
\end{proof}
Finally, we record a smooth compactness theorem for smooth properly embedded self-shrinkers. This is essentially shown in \cite{CM}, but 
we  
sketch an argument for completeness.
\begin{prop} \label{SmoothCpctProp}
There exists $R_0\geq 4$ so that: if $R\geq R_0$ and $\Sigma_i \subset 
B_{2R}\subset \Real^{3}$ is a sequence of smooth properly embedded self shrinkers satisfying
\begin{equation*}
\begin{array}{ccc} \mathcal{H}^2(B_4\cap \Sigma_i)\leq C & \mbox{and}&
 \mathrm{gen}(B_{2R}\cap \Sigma_i )\leq g
 \end{array}
\end{equation*}
for fixed $C>0$ and $g\geq 0$,
then, up to passing to a subsequence, the $\Sigma_i$ converge in $C^\infty_{loc}(B_{R})$ with multiplicity one to a smooth properly 
embedded self-shrinker $\Sigma_{\infty}\subset B_{R}$.
Here $\mathrm{gen}(\Sigma_i \cap B_{2R})$ is the sum of the genera of each component of $\Sigma_i \cap B_{2R}$.
\end{prop}
\begin{proof}
 Take $R_0=2^7$. 
Combining the area estimate of Lemma \ref{VolGrowLem}, the local Gauss-Bonnet estimate \cite[Theorem 3]{I} and 
\cite[Theorem 3]{W}, up to passing to a 
subsequence, there are a finite set of points $\set{p_1, \ldots, p_n}\subset B_R$ so that the $\Sigma_i\cap B_R$ converge in 
$C^\infty_{loc}(B_R\backslash \set{p_1, \ldots p_n})$ with finite multiplicity to some smooth properly embedded surface 
$\Sigma_\infty\subset B_{R}\backslash \set{p_1, \ldots, p_n}$.  Moreover, the surface $\Sigma_\infty$ extends smoothly to a smooth 
properly embedded 
surface (which we continue to denote by $\Sigma_\infty$) in $B_R$.
If the convergence is with multiplicity greater than one, then the argument of \cite[Proposition 3.2]{CM} shows the existence of a 
smooth positive solution to  $L_{\Sigma_\infty} u=0.$ That is, $\lambda(L_{\Sigma_\infty})\geq 0$, which contradicts Lemma \ref{InstabLem}. Finally, by Brakke's theorem, as the convergence is with 
multiplicity one, the set of singular points is empty.
\end{proof}
As a consequence, we observe rigidity phenomena for $\Real^2$, $\mathbb{S}^1\times \Real$ and $\mathbb{S}^2$:
\begin{cor} \label{cor:shrinkergap}
Suppose that $\mu$ is a Radon measure on $\Real^3$ so that $\mu= m'[\Sigma]$, where $m'\geq 1$ is an integer and $\Sigma$ is a smooth properly 
embedded self-shrinker. Given 
$C>0$ and $g\geq 0$, there exist $R=R(C)$ and $\epsilon_0=\epsilon_0(C,g)>0$ so that:
if 
\begin{equation*}
\begin{array}{cc} \mathcal{H}^2(B_4\cap \Sigma)\leq C, &
 \mathrm{gen}(\Sigma\cap B_R)\leq g,
 \end{array}
\end{equation*}
and
\begin{equation*}
 d_V(\mu, m[\mathbb{S}^k\times \Real^{2-k}])\leq \epsilon_0,
\end{equation*}for $0\leq k \leq 2$ and $m>0$,
then, $m'=m$ and, up to a rotation of $\Real^3$,  $\Sigma=\mathbb{S}^k\times \Real^{2-k}$.
\end{cor}
\begin{proof}
 Take $R\geq 2\max\set{R_1(2),R_0}$, where $R_1(2)$ is given by Corollary \ref{BrakkeCor} and $R_0$ is given by Proposition 
\ref{SmoothCpctProp}. We argue by contradiction.  Let $\mu_i=m_i [\Sigma_i]$ be a sequence of Radon measures, where 
the $\Sigma_i$ are  properly embedded self-shrinkers none of which differ from $\mathbb{S}^k\times \Real^{2-k}$ by a rotation and which 
satisfy the area and genus estimates and 
$$ d_V(\mu_i, m[\mathbb{S}^k\times \Real^{2-k}]) \to 0.$$   Up to passing to a subsequence, Proposition 
\ref{SmoothCpctProp} implies that the $\Sigma_i$ converge in  $C^\infty_{loc}(B_{R/2})$ with multiplicity one to
$\Sigma_\infty$, a properly embedded smooth shrinker in $B_{R/4}$. Furthermore, the supports of the $\mu_i$ converge to $\mathbb{S}^k\times 
\Real^{2-k}$ as sets and so $\Sigma_\infty\subset \mathbb{S}^k\times \Real^{2-k}$. Finally, as the supports of the $\mu_i$ converge to 
$\mathbb{S}^k\times \Real^{2-k}$ and the densities, $m_i$, of the $\mu_i$ are positive integers, the $m_i$ converge to $m$ and so, by 
passing to a further subsequence, $m_i=m$.

If $\Sigma_\infty$ is a sphere, then by \cite{Huisken} and the nature of the convergence, $\Sigma_i=\Sigma_\infty$ for $i$ sufficently 
large. Likewise, 
if $\Sigma_\infty$ is a piece of a cylinder, then by \cite{CIM} and the nature of the convergence, the $\Sigma_i$ are cylinders for $R$ and $i$
sufficently large.  Finally,  if $\Sigma_\infty$ is a piece of a plane, then, as $i\to \infty$,
\begin{equation*}
 \mathcal{H}^2(\Sigma_i\cap B_{R/2})\to \mathcal{H}^2(\Real^2\cap B_{R/2})
\end{equation*}
and
\begin{equation*}
 \int_{\Sigma_i\cap B_{R/2}} \Phi \to \int_{\Real^2\cap B_{R/2}} \Phi.
\end{equation*}
Hence, by Corollary \ref{BrakkeCor}, $\Sigma_i$ is a plane for large $i$.  All cases yield a contradiction.
\end{proof}

\section{Generic Uniqueness} 
Let $M_t$, $0\leq t<T$, be a smooth embedded mean curvature flow in $\Real^3$ which, for all $x\in\Real^3$, 
$0\leq t<T$, and $r>0$ satisfies 
\begin{equation}\label{eq:area}
\begin{array}{ccc} \mathrm{gen}(M_t)\le g_0 & \mbox{and} & \mathcal{H}^2\left(M_t\cap B_r(x)\right)\le C_0 r^2 \end{array}
\end{equation}
for some positive integer $g_0$, and $C_0>0$. For instance, such $g_0$ and $C_0$ exist if $M_0$ is closed. Given such a flow, for fixed 
$x_0\in\Real^3$, the rescaled mean curvature flow is obtained from $M_t$ by setting $N_s=(T-t)^{-1/2}(M_t-x_0)$, 
$s=-\log(T-t)$. As shrinking solutions to the mean curvature flow are static solutions to the rescaled mean curvature flow, Theorem 
\ref{thm:tangentflow} is an immediate consequence of the following:
\begin{prop}\label{prop:unique}
Suppose that there exists $s_j\to\infty$ such that
\[
d_V\left([N_{s_j}],m[\mathbb{S}^k\times\Real^{2-k}]\right)\to 0
\] 
for some integer $m\ge 1$ and $0\leq k \leq 2$. There exists a rotation $O_s\in SO(3)$ for each $s$ such that
\[
d_V\left([N_s],m [O_s(\mathbb{S}^k\times\Real^{2-k})]\right)\to 0\ \text{as}\ s\to\infty.
\]
\end{prop}
\begin{proof}
First, by the assumption, given an integer $l\ge 2$, there exists $\tilde{J}_l$ such that for $j\ge \tilde{J}_l$,
\[
d_V\left([N_{s_j}],m[\mathbb{S}^k\times\Real^{2-k}])\right)<\frac{\eps_0}{2^l}.
\]

On the other hand, since $s_j\to\infty$, it follows from \cite[Theorem 2]{I} that there exists $J_l\ge\tilde{J}_l$ such that if 
$s\ge s_{J_l}$,
\[
d_V\left([N_s],\mu_s\right)<\frac{\eps_0}{2^l} \quad\mbox{and}\quad 
\mathcal{F}[N_s]-m\mathcal{F}[\mathbb{S}^k\times\Real^{2-k}]<\frac{\eps_0}{2^l},
\]
where $\mu_s=m_s[\Sigma_s]$ for some integer $m_s\ge 1$ and $\Sigma_s$ is a smooth properly embedded self-shrinker satisfying
\begin{equation*}
\begin{array}{ccc} \mathrm{gen}(\Sigma_s)\le g_0 & \mbox{and} & \mathcal{H}^2\left(\Sigma_s\cap B_r(x)\right)\le C_0 r^2 \end{array}
\end{equation*}
for all $x\in\Real^3$ and $r>0$. We may arrange so that $J_l$ is increasing in $l$.

For any $l$ fixed, $[N_s]$ depends on $s\in [s_{J_l},s_{J_{l+1}}]$ in a continuous manner, so there exists $\delta_l$ such that for $s\in 
[s_{J_l},s_{J_{l+1}}]$ and $\Delta s< \delta_l$,
\[
d_V\left([N_s],[N_{s+\Delta s}]\right)<\frac{\eps_0}{2^l}.
\]
Thus, for each $s\in [s_{J_l},s_{J_l}+\delta_l]$,
\[
d_V\left(\mu_s,m [\mathbb{S}^k\times\Real^{2-k}])\right)\le d_V\left(\mu_s, [N_s])\right) +d_V\left( [N_s], [N_{s_{J_l}}]\right)
\le\frac{2\eps_0}{2^l}<\eps_0,
\]
and therefore, by Corollary \ref{cor:shrinkergap}, $\Sigma_s=O_s(\mathbb{S}^k\times\Real^{2-k})$ for some $O_s\in SO(3)$. In particular, 
\[
d_V\left([N_s] ,m[O_s(\mathbb{S}^k\times\Real^{2-k})]\right)<\frac{\eps_0}{2^{l}} 
\]
for every $s\in [s_{J_l},s_{J_l}+\delta_l]$. Starting with $s_{J_l}+\delta_l$, we repeat the above argument and conclude that this 
inequality holds on the whole interval $[s_{J_l},s_{J_{l+1}}]$. The claim follows from the fact that $l$ is arbitrary.
\end{proof}

{\bf Acknowledgement.} The second author would like to thank the University of Cambridge for their hospitality during her visit in May and June 2013 while this project was initiated.

\end{document}